\documentclass[11pt,reqno]{amsart}
\usepackage{amssymb}
\usepackage{amsmath}
\usepackage{mathrsfs}
\usepackage{amssymb, amsmath, amsfonts}

\usepackage{mathabx}

\usepackage{color}

\usepackage[hidelinks]{hyperref}
\allowdisplaybreaks

\makeatletter
\def\tank#1{\protected@xdef\@thanks{\@thanks
 \protect\footnotetext[0]{#1}}}
\def\bigfoot{

 \@footnotetext}
\makeatother

\newcommand{\ea}{\end{array}}

\allowdisplaybreaks
\numberwithin{equation}{section}
\newtheorem{theorem}{Theorem}[section]
\newtheorem{lemma}{Lemma}[section]
\newtheorem{proposition}{Proposition}[section]

\newtheorem{definition}{Definition}[section]

\newtheorem{condition}{Condition}[section]

\def\beq{\begin{equation}}
\def\nneq{\end{equation}}

\def\bthm{\begin{theorem}}
\def\nthm{\end{theorem}}

\def\blem{\begin{lemma}}
\def\nlem{\end{lemma}}
\def\bprf{\begin{proof}}
\def\nprf{\end{proof}}
\def\bprop{\begin{prop}}
\def\nprop{\end{prop}}
\def\brmk{\begin{rem}}
\def\nrmk{\end{rem}}

\def\bexa{\begin{exa}}
\def\nexa{\end{exa}}
\def\bcor{\begin{cor}}
\def\ncor{\end{cor}}

\def\e{\varepsilon}

\def\RR{\mathbb{R}}

\def\e{{\varepsilon}}

\title[]{Large deviation principle for reflected SPDE on infinite spatial domain }

\author{Ran Wang}
\address[]{Ran Wang, School of Mathematics and Statistics,  Wuhan University,  Wuhan, 430072,
China.}
\email{rwang@whu.edu.cn}

 \author{Beibei Zhang}
\address[]{Beibei Zhang, School of Mathematics and Statistics,  Wuhan University,  Wuhan, 430072,
China.}
\email{zhangbb@whu.edu.cn}

\date{}
\begin{document}
\maketitle

 \noindent {\bf Abstract}
 We study a large deviation principle for a reflected stochastic partial  differential equation on infinite spatial domain. A new sufficient condition for the weak convergence criterion  proposed by Matoussi, Sabbagh and Zhang ({\it Appl. Math. Optim.} 83: 849-879, 2021)  plays an important role in the proof.
 \vskip0.3cm

\vskip0.3cm

 \section{Introduction}

Consider the following  stochastic partial differential equation (SPDE for short)   with reflection:
\begin{equation}\label{SPDE}
\left\{\begin{split}
   &\frac{\partial u^{\e}}{\partial t}(t, x)= \frac{\partial^2 u^{\e}}{ \partial x^2} (t,x) + f(x,u^{\e}(t,x))+\sqrt{\e} \sigma(x,u^{\e}(t,x)) \dot W(t, x) +\eta^{\e}(dx,dt);\\
   &u^{\e}(0,\cdot)=u_0;\\
    &u^{\e}(t,0)=0,
 \end{split}\right.
\end{equation}
  where ${\e}>0$,  $(t,x)\in \mathbb R_+^2$;
  $\dot W(t,x)$   denotes the space-time white noise defined on a complete probability space
 $(\Omega, \mathcal F, \{\mathcal F_t\}_{t\ge0}, \mathbb P)$, with  $\mathcal F_t=\sigma\{W(s, x): 0\le s\le t, x\ge 0\}$; $u_0$ is a non-negative continuous function on $\mathbb R_+$; $\eta^{\e}$ is a random measure on $\mathbb R_+^2$ to keep $u^{\e}$ nonnegative and it is a part of the solution pair $(u^{\e}, \eta^{\e})$.

   When  the spatial domain is a finite interval $[0,1]$,     the  SPDE  with reflection was first studied  by Nualart and Pardoux  \cite{NP1992}, and they proved the existence and uniqueness of the solution in the case  of $\sigma(\cdot)=1$.  When the  volatility coefficients  $\sigma$  is Lipschitz with linear growth,    Donati-Martin and  Pardoux  \cite{DMP1993} proved the existence of the solution 
  to the equation,  Xu and Zhang \cite{XZ2009} proved its uniqueness.
Many interesting and important properties have been investigated, for example,    the   reversible probability measure   in \cite{FO2001,Zambotti2001}; the hitting properties of solutions  in \cite{DMZ2006}; the H\"older continuity in \cite{DZ2013};  the invariant measures  in \cite{Kalsi2020} and \cite{YangZ2014}; the large deviation principle (LDP for short)  in \cite{XZ2009} and \cite{YZ2020}; the dimension-free Harnack inequalities  and log-Harnack inequalities  in \cite{Zhang2010} and \cite{X2022}; the hypercontractivity  in \cite{X2019},    and so on.   The reflected SPDEs with fractional noises were studied in \cite{WJW2020, YZ2020}.    We would like to refer the reader to  \cite{Zambotti2017} and references therein for more information on reflected SPDEs.

 When the spatial domain is an infinite interval  $\mathbb R$ or $\mathbb R_+$,   Otobe  \cite{Oto2002}     proved  the existence and  uniqueness  in the case  of $\sigma(\cdot)=1$, and  proved  the existence when $\sigma$  is Lipschitz with linear growth. Uniqueness has also been shown by Hambly and Kalsi \cite{HK2019} for the equation on an unbounded domain provided that $\sigma$ satisfies a Lipschitz condition, with a Lipschitz coefficient which decays exponentially fast in the spatial variable.

 The purpose of this paper is to obtain the LDP for the solution of the reflected SPDE \eqref{SPDE} on an unbounded domain.

\begin{condition}\label{con 1}   Assume that  there are some constants $C_{1,i}>0$, $i=1,2,3$, $R>0,\delta\geq0$ and $r\in \mathbb R$, $f$ and  $\sigma $ are some measurable mappings from  $\mathbb R_+\times\mathbb R$ to $ \mathbb R$ satisfying  the following conditions:
\begin{itemize}
\item[(I)] For every $x\in [0,\infty)$, $u,v\in \mathbb R$,
$$
|f(x,u)-f(x,v)|\le C_{1,1}|u-v|.
$$
\item[(II)] For every $x\in \mathbb [0,\infty)$, $u \in \mathbb R$,
$$
|f(x,u) |\le C_{1,2}\left(e^{r x}+|u|\right).
$$
\item[(III)] For every $x\in [0,\infty)$, $u,v\in \mathbb R$,
$$
|\sigma(x,u)-\sigma(x,v)|\le C_{1,3}e^{-\delta x} |u-v|.
$$
\item[(IV)] For every $x\in [0,\infty)$, $u\in \mathbb R$,
$$
|\sigma(x,u) |\le R  e^{-\delta x  }\left(e^{r x }+ |u|\right).
$$
\end{itemize}
\end{condition}

Fix $r$ in $\mathbb R$.  We say that $u:\mathbb R_+\rightarrow\mathbb R$ is in the space $\mathcal L_r$,  if
\begin{align}\label{eq Lr}
\|u\|_{\mathcal L_r}:=\sup_{x\ge 0}e^{-rx} |u(x)|<\infty.
\end{align}
We say that $u\in \mathcal C_r$ if $u\in \mathcal L_r$ and $u$ is continuous; $u:[0,T]\times \mathbb R_+ \rightarrow\mathbb R$ is in the space $\mathcal C_r^T$, if $u$ is continuous and
\begin{align}\label{eq CrT}
\|u\|_{\mathcal C_{r}^T}:=\sup_{t\in [0,T]}\sup_{x\ge0} e^{-rx}   |u(t,x)|<\infty.
\end{align}
For any $r\in  \RR$, let
\begin{align}\label{eq LrT}
\|u\|_{t, \mathcal L_{r}}:=\sup_{x\ge0} e^{-rx}   |u(t,x)|.
\end{align}

The following definition is taken  from \cite{HK2019}.
\begin{definition}\cite[Definition 4.1]{HK2019}\label{def solution}
  A pair $(u^{\e},\eta^{\e})$  is called a solution of Eq.\,\eqref{SPDE}, if
\begin{itemize}
\item[(i)]  $u^{\e}$ is a continuous $\mathcal F_t$-adapted process taking values in $\mathcal C_{r}^T$ and $u^{\e}(t,x)\ge0$, a.s.
\item[(ii)] $\eta^{\e}$ is a random measure on $\mathbb R_+^2$ such that
 \begin{itemize}
\item[(a)] for every  measurable mapping $\psi:[0,\infty)\times [0,\infty)\rightarrow\mathbb R$,
$$
\int_0^{t}\int_0^{\infty}\psi(s,x)\eta^{\e}(dx,ds) \ \ {\text is}\,\, \mathcal F_t\text{-measurable}.
$$
\item[(b)]  $\int_0^{\infty}\int_0^{\infty} u^{\e}(s,x)\eta^{\e}(dx,ds)=0$, \   a.s.
 \end{itemize}
\item[(iii)] $(u^{\e},\eta^{\e})$ solves the parabolic SPDE in the following sense:
  for every $\varphi\in C_c^{2}(\mathbb R_+) $ with $\varphi(0)=0$,
\begin{equation}
\begin{split}
 \int_0^{\infty}u^{\e}(t,x)\varphi(x)dx =&\,\int_0^{\infty}u_0(x)\varphi(x)dx
+\int_0^t\int_0^{\infty}u^{\e}(s,x)\frac{\partial^2\varphi}{\partial x^2}(x)dxds\\
&\,  +\int_0^t\int_0^{\infty}   f(x, u^{\e}(s,x)) \varphi(x) dxds\\
 &\,+\sqrt{\e}\int_0^t\int_0^{\infty} \sigma(x, u^{\e}(s,x))\varphi(x) W(dx,ds)\\
  &\, +\int_0^t\int_0^{\infty}\varphi(x)\eta^{\e}(dx,ds),\ \ \text{a.s.}
\end{split}
\end{equation}
\end{itemize}
\end{definition}

According to     Hambly and Kalsi \cite{HK2019}, we know the following result about the existence and uniqueness of the solution to Eq.\,\eqref{SPDE}.
\begin{proposition}\cite[Theorem 4.3 and Remark 4.10]{HK2019}
Assume that Condition \ref{con 1} holds with some $\delta>0$. Then Eq.\,\eqref{SPDE} admits a unique solution $u^\e$ in $\mathcal C_r^T$. Moreover, $\mathbb E\left[\|u^\e\|_{\mathcal C_r^T}^p\right]<\infty$ for any $p\geq1$.
 \end{proposition}

There are some literature about the LDPs for SPDEs with reflections,   e.g.,  Xu and Zhang \cite{XZ2009}, Yang and Zhou \cite{YZ2020}, Wang et al. \cite{WJW2020}, etc. Those works are forced on the  situation  when  the spatial domain is a finite interval with the sup-norm.   The purpose of this paper is to obtain the LDP result when  the spatial domain is an infinite interval with the weighted sup-norm $\|\cdot\|_{\mathcal C_r^T}$  defined by \eqref{eq CrT}. There is some extra complexity introduced in the case of an infinite spatial domain.  Here, we  adopt  a new sufficient condition for the LDP (see Condition \ref{cond1} below)   proposed by Matoussi et al. \cite{MSZ}, and we use some powerful estimates obtained in  Hambly and Kasli \cite{HK2019} (see Appendix of this paper) in the proof.

The rest of this paper is organized as follows. In Section 2, we first recall a sufficient condition of the weak convergence criterion for the LDP given in \cite{MSZ}, then we formulate the main result of the present paper.
In Section 3, the skeleton equation is  studied. Sections 4 and 5 are devoted to verifying   two conditions for the weak convergence criterion. Finally, we give some useful estimates borrowed from \cite{HK2019} in Appendix.

\section{A criterion for large deviations and main result}
In this section, we first recall the definition of the LDP and the weak convergence criterion in \cite{MSZ}.  Then we state the main result of this paper.

\subsection{A criterion for LDP}\label{A Criteria for Large Deviations}
Let $(\Omega,\mathcal{F},\{\mathcal{F}_t\}_{0\leq t\leq T},\mathbb{P})$ be a complete probability space, and the increasing family $\{\mathcal{F}_t\}_{0\leq t\leq T}$ satisfies the usual conditions. Let $\mathcal{E}$ be a Polish space with the Borel $\sigma$-field $\mathcal{B}(\mathcal{E})$.
 \begin{definition}\label{Dfn-Rate function}
     A function $I: \mathcal{E}\rightarrow[0,\infty]$ is called a rate function on
       $\mathcal{E}$,
       if for each $M<\infty$ the level set $\{y\in \mathcal{E}:I(y)\leq M\}$ is a compact subset of $\mathcal{E}$.
    \end{definition}
    \begin{definition}  \label{d:LDP}
       Let $I$ be a rate function on $\mathcal{E}$. The sequence
       $\{u^\e\}_{\e>0}$
       is said to satisfy  a large deviation principle on $\mathcal{E}$ with the rate function $I$, if the following two
       conditions
       hold:
\begin{itemize}
\item[(a)]   For each closed subset $F$ of $\mathcal{E}$,
              $$
                \limsup_{\e\rightarrow 0}\e \log\mathbb{P}(u^\e\in F)\leq\inf_{y\in F}I(y).
              $$

        \item[(b)]  For each open subset $G$ of $\mathcal{E}$,
              $$
                \liminf_{\e\rightarrow 0}\e \log\mathbb{P}(u^\e\in G)\geq-\inf_{y\in G}I(y).
              $$
              \end{itemize}
    \end{definition}
  The Cameron-Martin space associated with the Brownian sheet $\{W(t,x), t\in [0,T], x\in [0,\infty)\}$ is given by
\begin{align}\label{space H}
\mathcal H=\left\{g=\int_0^{\cdot}\int_0^{\cdot} \dot g(s,x) dxds;\,\, \int_0^{T}\int_0^{\infty} \dot g^2(s,x)dxds<\infty \right\}.
\end{align}
Let $$\|g\|_{\mathcal H}:=\left(\int_0^{T}\int_0^{\infty} \dot g^2(s,x)dxds\right)^{\frac{1}{2}},$$ and
\begin{align}\label{definitionSN}
S_N:=\{g\in\mathcal H; \|g\|_{\mathcal H}\leq N\}.
\end{align}
The set  $S_N$ is endowed with the weak convergence topology. Let $\mathcal{A}$ denote the class of all $\mathcal H$-valued $\mathcal F_t$-predictable processes, and let
$$\mathcal{A}_N:=\{\phi\in\mathcal A; \phi(\omega)\in S_N,\,  \mathbb{P}{\text-a.s.}\}.$$

  \begin{condition}\label{cond1} For any $\e>0$, let $\Gamma^{\e}$ be a measurable mapping from $C([0,T]\times\RR_+)$ into $\mathcal{E}$. Suppose that there exists a measurable mapping $\Gamma^{0}: C([0,T]\times\RR_+)\rightarrow\mathcal{E}$ such that
\begin{itemize}
\item[(a)] For every $N<+\infty$, let $g_n$, $g\in S_N$ be such that $g_n\rightarrow g$ weakly as $n\rightarrow\infty$. Then,  $\Gamma^0\left(\int_0^{\cdot}\int_0^{\cdot}  \dot{g}_n(s,x)dxds\right)$ converges to  $\Gamma^0\left(\int_0^{\cdot}\int_0^{\cdot}  \dot{g}(s,x)dxds\right)$ in the space $\mathcal{E}$.

   \item[(b)] For every $N<+\infty$, $\{g^\varepsilon\}_{\varepsilon>0}\subset \mathcal A_N$  and $\theta>0$,
    $$\lim_{\varepsilon\rightarrow 0}\mathbb P\left(\rho(Y^{\varepsilon}, Z^{\varepsilon})>\theta\right)=0, $$
     where  $Y^{\varepsilon}=\Gamma^\varepsilon\left(W+\frac{1}{\sqrt\varepsilon} \int_0^{\cdot}\int_0^{\cdot}  \dot{g}^{\varepsilon}(s,x)dxds\right), Z^{\varepsilon}=\Gamma^0\left(\int_0^{\cdot}\int_0^{\cdot} \dot{g}^{\varepsilon}(s,x)dxds\right)$ and $\rho(\cdot, \cdot)$ stands for the metric in the space $\mathcal{E}$.

 \end{itemize}
  \end{condition}
  Let   $I$ be defined  by
\beq\label{rate function}
I(h):=\inf_{\left\{g\in \mathcal{H};h=\Gamma^0\left(\int_0^{\cdot}\int_0^{\cdot}\dot{g}(s,x)dxds\right)\right\}}  \left\{\frac{1}{2}\|g\|^2_{\mathcal{H}}\right\},\  h\in\mathcal{E},
\nneq
with the convention $\inf\emptyset=\infty$.

Recall the following result from  Matoussi et al.  \cite{MSZ}.
\begin{theorem}\label{thm MSZ} (\cite[Theorem 3.2]{MSZ}) Suppose that Condition \ref{cond1} holds. Then,  the family $\{\Gamma^{\varepsilon}(W)\}_{\varepsilon>0}$ satisfies a large deviation principle with the rate function  $I$ defined by \eqref{rate function}.
\end{theorem}

\subsection{Main result}\label{Main result}
To state our main result, we need to introduce a map $\Gamma^0$.
Given $g\in\mathcal{H}$,  consider the following deterministic integral equation (the skeleton equation):
\begin{equation}\label{eq skeleton}
\left\{\begin{split}
   &\frac{\partial u^{g}}{\partial t}(t, x)= \frac{\partial^2 u^{g}}{ \partial x^2} (t,x) + f(x,u^{g}(t,x))+  \sigma(x,u^{g}(t,x)) \dot{g}(t,x) +\eta^{g}(dx,dt);\\
&u^{g}(0,\cdot)=u_0;\\
 &u^{g}(\cdot,0)=0.
 \end{split}\right.
\end{equation}
 Analogously to Definition \ref{def solution}, a pair of $(u^g, \eta^g)$ is called a solution of   Eq.\,\eqref{eq skeleton}, if it satisfies the following conditions:
\begin{itemize}
\item[(i)] $u^g$ is in $\mathcal C_{r}^T$ and $u^g(t,x)\ge0$.
\item[(ii)] $\eta^g$ is a   measure on $\RR_+^2$ such that
$\int_0^{\infty}\int_0^{\infty} u(t,x)\eta^g(dx,dt)=0$.
\item[(iii)] $(u^g,\eta^g)$ solves the parabolic  PDE in the following sense:  for every $\varphi\in C_c^{2}(\mathbb R_+)$ with $\varphi(0)=0$, and for every $t\ge0$,
\begin{equation}
\begin{split}
 \int_0^{\infty}u^g(t,x)\varphi(x)dx=&\,\int_0^{\infty}u_0(x)\varphi(x)dx +\int_0^t\int_0^{\infty}u^g(s,x)\frac{\partial^2\varphi}{\partial x^2}(x)dxds\\
 &\,+\int_0^t\int_0^{\infty}   f(x, u^g(s,x)) \varphi(x) dxds\\
&\, +\int_0^t\int_0^{\infty} \sigma(x, u^g(s,x))\varphi(x) \dot{g}(s,x)dxds  \\
&\,  +\int_0^t\int_0^{\infty}\varphi(x)\eta^g(dx,ds).
\end{split}
\end{equation}
\end{itemize}
 By Proposition \ref{thm solu skeleton} below, Eq.\,\eqref{eq skeleton} admits a unique solution $u^g\in \mathcal C_r^T$ under Condition \ref{con 1}.
 For any $g\in\mathcal{H}$, define
 \begin{align}\label{eq Gamma0}
 \Gamma^0\left(\int_0^{\cdot}\int_0^{\cdot}  \dot{g}(s,x)dxds\right):=u^g.
 \end{align}

The following theorem is the main result of this paper.
\begin{theorem}\label{thm LDP}  Assume  that Condition \ref{con 1} holds with some $\delta>0$. Then the family $\{u^{\e}\}_{\e>0}$ in Eq.\,\eqref{SPDE}  satisfies an LDP in the space  $\left(\mathcal C_r^T,\|\cdot\|_{\mathcal C_r^T}\right)$ with the rate function $I$ given by
\begin{equation}\label{eq rate1}
I(h):=\inf_{\left\{g\in \mathcal H, h=\Gamma^0\left(\int_0^{\cdot}\int_0^{\cdot}  \dot{g}(s,x)dxds\right)\right\} }\left\{ \frac12 \int_0^{T}\int_0^{\infty}\dot{g}^2(s,x)dxds \right\}, \,\,h\in\Gamma^0( \mathcal H),
\end{equation}
with the convention $\inf\emptyset=\infty$.
\end{theorem}
\begin{proof}
According to Theorem \ref{thm MSZ}, it suffices to show that the conditions  (a) and (b) in Condition \ref{cond1} are satisfied. Condition (a) will be proved in Proposition \ref{3.2theorem}, and Condition (b) will be verified in Proposition \ref{proposition-4-3}. The proof is complete.
\end{proof}
\section{The skeleton equation}

\subsection{The deterministic obstacle problem on $\mathbb R_+$}

Recall the following result for the deterministic obstacle problem on $\mathbb R_+$ from Hambly and Kalsi \cite{HK2019}.

\begin{definition}\cite[Definition 2.5]{HK2019}\label{def3.1teach} Fix $r\in \mathbb R$. Let $v\in \mathcal C_r^T$ be such that $v(t,0)=0$ for every $t\in [0,T]$ and $v(0,\cdot)\le 0$. We say that the pair $(z,\eta)$ satisfies the heat equation with the obstacle $v$ and exponential growth $r$ on $\mathbb R_+$, if
\begin{itemize}
\item[(i)] $z\in \mathcal C_r^T$,  $z(t,0)=0, z(0,x)=0$ and $z\ge v$.

\item[(ii)] $\eta$ is a measure on $(0, \infty)\times [0,T]$ such that  $\int_0^{T}\int_0^{\infty}\left(z(s,x)-v(s,x) \right)\eta(dx, ds)=0$.
\item[(iii)] $z$ weakly solves the PDE
\begin{align}\label{eq DOP}
\frac{\partial z}{\partial t}=\frac{\partial^2 z}{\partial x^2}+\eta.
\end{align}
That is, for every $t\in [0,T]$ and every $\varphi\in  C_c^{\infty} (\mathbb R_+)$ with $\varphi(0)=0$,
$$
\int_0^{\infty} z(t,x)\varphi(x)dx=\int_0^t \int_0^{\infty} z(s,x)\varphi''(x)dxds+\int_0^t \int_0^{\infty} \varphi(x)\eta(dx,ds).
$$

\end{itemize}

\end{definition}

\begin{theorem}\cite[Theorem 2.6]{HK2019} \label{thm DOP} For every $r\in \mathbb R$ and every $v\in \mathcal C_r^T$ such that $v(t,0)=0$ for every $t\in [0,T]$ and $v(0,\cdot)\le 0$,
there exists a unique solution $(z, \eta)$ to the heat equation on $\mathbb R_+$ with Dirichlet boundary condition and the obstacle $v$. Furthermore, if $v_1, v_2\in \mathcal C_r^T$, then
\begin{align}\label{eq comp}
\|z_1-z_2\|_{\mathcal C_r^T}\le C_{3, 1}\|v_1-v_2\|_{\mathcal C_r^T},
\end{align}
where $C_{3, 1}=C_{3, 1}(r,T)\in(0,\infty)$ and $z_i$ is the solution to the obstacle problem corresponding to $v_i$, $i=1,2$.
\end{theorem}

\subsection{The skeleton equation}
 In this part, we prove the existence and uniqueness of the solution to the skeleton equation \eqref{eq skeleton}. The following notations will be used throughout the rest of the paper.
 \begin{definition}
 Let $G(t,x,y)$ be the Dirichlet heat kernel on $[0,\infty)$,  that is, for any $t>0, x, y\in\mathbb R_+$,
\begin{equation}\label{eq kernel}
G(t,x,y):=\frac{1}{\sqrt{4\pi t}}\left[\exp\left(-\frac{(x-y)^2}{4t} \right)-\exp\left(-\frac{(x+y)^2}{4t} \right)\right].
\end{equation}
For any $r\in \mathbb R$, let
\begin{equation}\label{eq Gr}
G_r(t,x,y):=e^{-r(x-y)}G(t,x,y), \,\,t>0,\,\,x,\,y\in\RR_+.
\end{equation}
\end{definition}
\begin{proposition}\label{thm solu skeleton} Assume that $g\in \mathcal H$ and Condition \ref{con 1} holds with some $\delta=0$. Then
  Eq.\,\eqref{eq skeleton} admits a unique solution in the space  $\left(\mathcal C_r^T,\|\cdot\|_{\mathcal C_r^T}\right)$.
  \end{proposition}
\begin{proof}
{\bf  Existence.}   Assume that $g\in S_N$ for some $N\ge1$. We use a Picard argument to prove the existence  of the solution.
 Let
\begin{equation}\label{eq v1}
\begin{split}
v_1^g(t,x)=&\, \int_0^{\infty}G(t, x,y) u_0(y)dy +\int_0^t\int_0^{\infty}G(t-s, x,y)f(y, u_0(y)) dyds\\
&\,\,\, + \int_0^t\int_0^{\infty}G(t-s, x,y)\sigma(y, u_0(y)) \dot g(s,y) dyds,
\end{split}
\end{equation}
with $v_1^g(\cdot,0)=0$.
 Denote     the solution of Eq.\,\eqref{eq DOP}  with the obstacle $v=-v_1^g$ by  $(z_1^g, \eta_1^g)$ and  set $u_1^g=z_1^g+v_1^g$. Then $(u_1^g, \eta_1^g)$ is a solution of the following reflected PDE:
\begin{equation}\label{eq u1}
   \frac{\partial u_1^{g}}{\partial t}(t, x)= \frac{\partial^2 u_1^{g}}{ \partial x^2} (t,x) + f(x,u_0(x))+  \sigma(x,u_0(x))   \dot{g}(t,x) +\eta^{g}_1(dx,dt), 
 \end{equation}
  with $u_1^g(0,x)=u_0(x)$ and $u_1^g(t,0)=0$ for any $t\in[0,T], x\in\RR_+$.
 
 Inductively, put
\begin{equation}\label{eq vn}
\begin{split}
v_n^g(t,x)=&\, \int_0^{\infty}G(t, x,y) u_0(y)dy +\int_0^t\int_0^{\infty}G(t-s, x,y)f(y, u_{n-1}^g(s,y)) dyds\\
&\,\,\, + \int_0^t\int_0^{\infty}G(t-s, x,y)\sigma(y, u_{n-1}^g(s,y))  \dot{g}(s,y) dyds,
\end{split}
\end{equation}
  and define $(z_n^g, \eta_n^g)$ as  the solution of Eq.\,\eqref{eq DOP} with $v=-v_n^g$. Let  $u_n^g=z_n^g+v_n^g$. Then $(u_n^g, \eta_n^g)$ is the solution of the following reflected PDE:
\begin{equation}\label{eq un}
   \frac{\partial u_n^{g}}{\partial t}(t, x)= \frac{\partial^2 u_n^{g}}{ \partial x^2} (t,x) + f(x,u^{g}_{n-1}(t,x))+  \sigma(x,u_{n-1}^{g}(t,x)) \dot{g}(t,x) +\eta^{g}_n(dx,dt), 
 \end{equation}
 with $u_n^g(0,x)=u_0(x)$ and $u_n^g(t,0)=0$ for any $t\in[0,T]  ,x\in\RR_+$.

By Theorem \ref{thm DOP}, we obtain that for $n\geq2$,
\begin{equation}\label{u_ng_u_n-1g}
\begin{split}
& \|u_n^{g}-u_{n-1}^{g}\|_{\mathcal C_r^T}\\
\le&\,  2 C_{3, 1}\|v_n^{g}-v_{n-1}^{g}\|_{\mathcal C_r^T}\\
\le & \,  2  C_{3, 1} \sup_{t\in [0,T]}\sup_{x\ge0} e^{-rx} \left|\int_0^t\int_0^{\infty}G(t-s, x,y) \left( f(y, u_{n-1}^g(s,y))- f(y, u_{n-2}^g(s,y))  \right) dyds \right|\\
&\,+ 2  C_{3, 1} \sup_{t\in [0,T]}\sup_{x\ge0} e^{-rx} \left|\int_0^t\int_0^{\infty}G(t-s, x,y) \left( \sigma(y, u_{n-1}^g(s,y))- \sigma(y, u_{n-2}^g(s,y))  \right)  \dot{g}(s,y) dyds \right|\\
=:&\,2C_{3, 1} \left[I_1(T)+I_2(T)\right].
\end{split}
\end{equation}
By Lemma \ref{lem int}, Condition \ref{con 1} (I) and H\"older's inequality, we have that for any $p>4$, 
\begin{equation}\label{I_1p11}
\begin{split}
I_1 (T)^p \le &\, C^p_{6, 1}\left[\int_0^T\|f(\cdot,u_{n-1}^{g})-f(\cdot,u_{n-2}^{g})\|_{t, \mathcal L_r} dt \right]^p\\
\le &\,C^p_{6, 1}\cdot C^p_{1, 1}  \left[\int_0^T\|u_{n-1}^{g}-u_{n-2}^{g}\|_{t, \mathcal L_r} dt \right]^p\\
\le &\,   C^p_{6, 1}\cdot C^p_{1, 1}\cdot T^{p-1} \int_0^T\|u_{n-1}^{g}-u_{n-2}^{g}\|_{t, \mathcal L_r}^pdt.
\end{split}
\end{equation}
Since $g\in S_N$, by H\"older's inequality and Condition \ref{con 1} (III), we have that for any $p>4$, $t\in[0,T]$,
\begin{equation}\label{smallbound310}
\begin{split}
 &\left| e^{-rx} \int_0^t\int_0^{\infty}G(t-s, x, y) \left( \sigma(y, u_{n-1}^g(s,y))- \sigma(y, u_{n-2}^g(s,y))  \right) \dot{g}(s,y) dyds \right|^p\\
 =&\, \left|\int_0^t\int_0^{\infty}G_r(t-s, x, y) e^{-ry} \left( \sigma(y, u_{n-1}^g(s,y))- \sigma(y, u_{n-2}^g(s,y))  \right)   \dot{g}(s,y) dyds \right|^p\\
\le &\, C^p_{1, 3} \left|\int_0^t\int_0^{\infty}G_r(t-s, x, y)^2 e^{-2ry} \left|  u_{n-1}^g(s,y)- u_{n-2}^g(s,y)\right|^2 dyds   \right|^{\frac{p}{2}}\\
&\,\,\,\,\,\,\, \cdot \left| \int_0^t\int_0^{\infty}  \dot{g}^2(s,y) dyds \right|^{\frac{p}{2}}\\
\le & \, C^p_{1, 3}\cdot N^{p} \left| \int_0^t\left(\int_0^{\infty} G_r(t-s, x, y)^2dy\right)   \left\|u_{n-1}^g-  u^g_{n-2}\right\|^2_{s,\mathcal L_r}ds \right|^{\frac{p}{2}} \\
\le & C^p_{1, 3}\cdot N^{p}\cdot \left( \int_0^t\left(\int_0^{\infty} G_r(t-s, x, y)^2dy\right)^{\frac{p}{p-2}} ds\right)^{\frac{p-2}{2}}\cdot  \int_0^t\left\|u_{n-1}^g-  u^g_{n-2}\right\|^{p}_{s,\mathcal L_r}ds,
  \end{split}
  \end{equation}
where $G_r(t-s, x, y)$ is defined in \eqref{eq Gr}.
By Lemma \ref{6_1} (i), we have
\begin{equation}\label{I2pbound1}
\begin{split}
I_2  (T)^p \leq  C^p_{1, 3}\cdot C_{6,3}\cdot N^{p}\cdot T^{\frac{p-4}{4}}\cdot
\int_0^T\left\|u_{n-1}^g-  u^g_{n-2}\right\|^{p}_{t,\mathcal L_r}dt.
\end{split}
\end{equation}
By \eqref{u_ng_u_n-1g}, \eqref{I_1p11} and \eqref{I2pbound1}, there exists a positive constant $C_{3,2}=C_{3,2}(r, p, N,T)$ such that
\begin{align}\label{eq u n n-1}
\|u_n^{g}-u_{n-1}^{g}\|_{\mathcal C_r^T}^p\le &\, C_{3,2} \int_0^T\left\|u_{n-1}^g-  u^g_{n-2}\right\|^{p}_{t,\mathcal L_r}dt \notag\\
\le &\,  C_{3,2}^{n-1} \int_0^T\int_0^{t_1}\cdots \int_0^{t_{n-2}}\left\|u_{1}^g-  u^g_{0}\right\|^{p}_{t_{n-1},\mathcal L_r}dt_{n-1}\cdots dt_2dt_1\\
\le &\, C_{3,2}^{n-1} \left\|u_{1}^g-  u^g_{0}\right\|^{p}_{\mathcal C_r^T} \cdot \frac{T^{n-1}}{(n-1)!}.\notag
\end{align}
In particular, by using the same procedure, we have
\begin{align*}
\|u_1^g\|_{\mathcal{C}^T_r}<+\infty.
\end{align*}
Therefore, for any $m\ge n\ge 1$,
\begin{equation}\label{eq umn}
\begin{split}
\|u_m^{g}-u_{n}^{g}\|_{\mathcal C_r^T}\le &\, \left\|u_{1}^g-  u^g_{0}\right\|_{\mathcal C_r^T}  \cdot \sum_{k=n}^{m-1} \left(\frac{C_{3,2}^{k} T^k }{k!} \right)^{\frac1p}\\
&\longrightarrow 0,\,\,\,\text{as } n, m\rightarrow \infty.
\end{split}
\end{equation}
Hence, $\{u_n^g\}_{n\geq1}$ is Cauchy in ${\mathcal C_r^T}$, and its limit  is denoted by
\begin{align}\label{limitugung}
u^g(t,x)=\lim_{n\rightarrow\infty} u_n^g(t,x).
\end{align}
Next, we prove that $u^g$ is a solution to Eq.\,\eqref{eq skeleton}.

Let
\begin{equation}\label{eq tildevg}
\begin{split}
\tilde{v}^g(t,x)=&\, \int_0^{\infty}G(t, x,y) u_0(y)dy +\int_0^t\int_0^{\infty}G(t-s, x,y)f(y, u^g(s,y)) dyds\\
&\,\,\, + \int_0^t\int_0^{\infty}G(t-s, x,y)\sigma(s, u^g(s,y)) \dot{g}(s,y) dyds.
\end{split}
\end{equation}
Define $\tilde{u}^g=\tilde{v}^g+\tilde{z}^g$, where $\tilde{z}^g$, together with a measure $\tilde{\eta}^g$ on $(0, \infty)\times [0,T]$, solves the obstacle problem Eq.\,\eqref{eq DOP} with obstacle $-\tilde{v}^g$. Thus, by Definition \ref{def3.1teach}, we  have the following results:
\begin{itemize}
\item[(i)] $z\in \mathcal C_r^T$,  $\tilde{z}^g(t,0)=0, \tilde{z}^g(0,x)=0$ and $\tilde{z}^g\ge -\tilde{v}^g$, which implies that $\tilde{u}^g\ge0$.
\item[(ii)] For any $T>0$, $\int_0^{T}\int_0^{\infty}\left(\tilde{z}^g(s,x)+\tilde{v}^g(s,x) \right)\tilde{\eta}^g(dx, ds)=0$, which is equivalent to $$\int_0^{T}\int_0^{\infty}\tilde{u}^g(s,x)\tilde{\eta}^g(dx, ds)=0.$$
\item[(iii)] $\tilde{z}^g$ weakly solves the PDE
\begin{align}\label{eq DOP1}
\frac{\partial \tilde{z}^g}{\partial t}=\frac{\partial^2 \tilde{z}^g}{\partial x^2}+\tilde{\eta}^g.
\end{align}
\end{itemize}
Putting \eqref{eq tildevg} and \eqref{eq DOP1} together, we obtain  that for every $\varphi\in C_c^{2}(\RR_+)$ with $\varphi(0)=0$, and for every $t\ge0$,
\begin{align*}
 \int_0^{\infty}\tilde{u}^g(t,x)\varphi(x)dx=&\,\int_0^{\infty}\tilde{u}^g(0,x)\varphi(x)dx
 + \int_0^t\int_0^{\infty}\tilde{u}^g(s,x)\frac{\partial^2\varphi}{\partial x^2}(x)dxds\\
 &\, +\int_0^t\int_0^{\infty}   f(x, \tilde{u}^g(s,x)) \varphi(x) dxds\\
&\, +\int_0^t\int_0^{\infty} \sigma(x, \tilde{u}^g(s,x))\varphi(x) \dot g(s,x)dxds  \\
  &\, +\int_0^t\int_0^{\infty}\varphi(x)\tilde{\eta}^g(dx,ds).
\end{align*}
As in the proof of \eqref{eq u n n-1},  by \eqref{limitugung},
 we    obtain that
\begin{align*}
\|u_n^{g}-\tilde{u}^g\|_{\mathcal C_r^T}^p\le&\, C_{3,2} \int_0^T\left\|u_{n-1}^g-  {u}^g\right\|^{p}_{t,\mathcal L_r}dt \\
\le &\, C_{3,2} T\left\|u_{n-1}^g-  {u}^g\right\|^p_{\mathcal C_r^T}\rightarrow 0,
\end{align*}
which implies that $\tilde{u}^g=u^g$. Thus,   $(\tilde{u}^g, \tilde{\eta}^g)$ is a solution to Eq.\,\eqref{eq skeleton}.

{\bf Uniqueness.} Given two solutions $(u_1^g, \eta_1^g)$ and $(u_2^g, \eta_2^g)$ with the same initial value, by the similar method used in the proof of   \eqref{eq u n n-1},
 we obtain that
 \begin{align*}
 \|u_1^{g}-u_2^{g}\|_{\mathcal C_r^T}\leq\,C_{3,2}\int^T_0 \left\|u_1^{g}-u_2^{g}\right\|_{\mathcal C_r^t}^pdt.
\end{align*}
By Gronwall's inequality, we have $u_1^{g}=u_2^{g}$.

The proof is complete.
\end{proof}

\section{Verification of Condition \ref{cond1} $(a)$}
In this section, we will show the continuity of the skeleton equation. Recall that $S_N$ is defined by \eqref{definitionSN}.
\begin{proposition}\label{3.2theorem} Assume  that Condition \ref{con 1} holds with some $\delta=0$.  Then the mapping $u^g: g\in S_N\longmapsto  u^g\in \mathcal C_r^T$ is   continuous.
 \end{proposition}
\begin{proof}
Let $g$, $g_n$, $n\geq1$, be in $S_N$ such that $g_n\rightarrow g$ weakly as $n\rightarrow\infty$. Let $u^{g_n}$ and $u^{g}$ be the solutions to Eq.\,\eqref{eq skeleton} associated with $g_n$ and $g$, respectively. It is sufficient to prove
\begin{align}\label{eq Gn_gCrT}
\lim_{n\rightarrow\infty}\|u^{g_n}-u^{g}\|_{\mathcal C_r^T}=0.
 \end{align}
Let
\begin{equation}\label{eq vgn}
\begin{split}
   v^{g_n}(t,x)=&\, \int_0^{\infty}G(t, x,y) u_0(y)dy +\int_0^t\int_0^{\infty}G(t-s, x,y)f(y, u^{g_n}(s,y)) dyds\\
&\,\,\, + \int_0^t\int_0^{\infty}G(t-s, x,y)\sigma(y, u^{g_n}(s,y))  \dot{g}_n(s,y) dyds,
 \end{split}
 \end{equation}
and
\begin{equation}\label{eq vg}
\begin{split}
   v^{g}(t,x)=&\, \int_0^{\infty}G(t, x,y) u_0(y)dy +\int_0^t\int_0^{\infty}G(t-s, x,y)f(y, u^{g}(s,y)) dyds\\
&\,\,\, + \int_0^t\int_0^{\infty}G(t-s, x,y)\sigma(y, u^{g}(s,y))  \dot{g}(s,y) dyds.
\end{split}
 \end{equation}
By Theorem \ref{thm DOP}, we obtain that
\begin{equation}\label{ug_nugI_1I_2I_3}
\begin{split}
& \|u^{g_n}-u^{g}\|_{\mathcal C_r^T}\\
\le&\,  2 C_{3,1}\|v^{g_n}-v^{g}\|_{\mathcal C_r^T}\\
\le & \,  2   C_{3,1} \sup_{t\in[0,T]}\sup_{x\ge0} e^{-rx} \left|\int_0^t\int_0^{\infty}G(t-s, x,y) \left( f(y, u^{g_n}(s,y))- f(y, u^{g}(s,y))  \right) dyds \right|\\
&\, \, \, + 2   C_{3,1}\sup_{t\in[0,T]}\sup_{x\ge0} e^{-rx} \left|\int_0^t\int_0^{\infty}G(t-s, x,y) \left( \sigma(y, u^{g_n}(s,y))- \sigma(y, u^g(s,y))  \right) \dot{g}_n(s,y) dyds \right|\\
&\, \, \, + 2   C_{3,1}\sup_{t\in[0,T]} \sup_{x\ge0} e^{-rx} \left|\int_0^t\int_0^{\infty}G(t-s, x,y) \sigma(y, u^{g}(s,y))\left( \dot{g}_n(s,y)-  \dot{g}(s,y)  \right)dyds \right|\\
\,=:& 2   C_{3,1}\left[K_1(T)+K_2(T)+K_3(T)\right].
\end{split}
\end{equation}
 By using the same method as that in the proofs of \eqref{I_1p11} and \eqref{I2pbound1},  we have that for any $p>4$,
\begin{align}
K_1 (T)^p
\le &\, C^p_{6,1}\cdot C^p_{1,1}\cdot T^{p-1} \int_0^T\|u^{g_n}-u^g\|_{\mathcal C_r^t}^pdt;\label{I_1estimate}\\
  K_2 (T)^p
\le &\,C^p_{1,3}\cdot C_{6,3}\cdot N^{p}\cdot T^{\frac{p-4}{4}}\int_0^T\left\|u^{g_n}-u^g\right\|_{\mathcal C_r^t}^pdt.  \label{I_2estimate}
\end{align}
Set
\begin{align*}
F_n(t,x):=\int_0^t\int_0^{\infty}G(t-s, x,y) \sigma(y, u^{g}(s,y))\left( \dot{g}_n(s,y)-  \dot{g}(s,y)  \right)dyds.
\end{align*}
Putting \eqref{ug_nugI_1I_2I_3}, \eqref{I_1estimate} and \eqref{I_2estimate} together, we have that there exists a positive  $C_{4,1}=C_{4,1}(r,p,N,T)$ such that
\begin{align*}
\|u^{g_n}-u^{g}\|^p_{\mathcal C_r^T}\leq\,C_{4,1} \int_0^T\|u^{g_n}-u^{g}\|^p_{\mathcal C_r^t}dt+2^pC^p_{3,1}\cdot\|F_n\|_{\mathcal C_r^T}^p.
\end{align*}
 By the Gronwall inequality, we have
\begin{align}\label{Gronadd4.7}
\|u^{g_n}-u^{g}\|^p_{\mathcal C_r^T}\leq&\,2^pC^p_{3,1}\cdot\|F_n\|_{\mathcal C_r^T}^p\cdot\left(1+C_{4,1}T\cdot e^{C_{4,1}T}\right).
\end{align}
Consequently, if \begin{align}\label{limFn0Crt}
 \lim_{n\rightarrow\infty}\|F_n\|_{\mathcal C_r^T}=0,
\end{align}
 then  \eqref{eq Gn_gCrT} holds.
Now, it remains to   prove \eqref{limFn0Crt}. 

By Condition \ref{con 1} (IV), H\"older's inequality and Lemma \ref{6_1} (i),  there exists $C_{4,2}=C_{6,3}^{\frac{2}{p}}\cdot R\cdot  T^{\frac{1}{2}}$ such that for  any fixed $x\in\RR_+$, $t\in[0,T]$,
\begin{align}\label{4-7est}
&\int_0^t\int_0^{\infty}G(t-s, x, y)^2  \left|\sigma(y, u^g(s,y))\right|^2dyds \notag\\
\leq&\,2R^2\cdot e^{2r x}\cdot\int_0^t\int_0^{\infty}G_r(t-s, x, y)^2\left(1+e^{-2r y}|u^g(s,y)|^2 \right)  dyds\notag\\
\leq&\,2R^2\cdot e^{2r x}\cdot\int_0^t\left(\int_0^{\infty}G_r(t-s, x, y)^2dy\right)\left(1+\|u^g\|^2_{s,\mathcal L_r} \right)ds\\
\leq&\, 2R^2\cdot e^{2r x}\cdot\left[\int_0^t\left(\int_0^{\infty}G_r(t-s, x, y)^2  dy\right)^{\frac{p}{p-2}}ds\right]^{\frac{p-2}{p}}\cdot \left[T^{\frac{2}{p}}+\left(\int^t_0\|u^g\|^p_{s,\mathcal L_r}ds\right)^{\frac{2}{p}}\right]\notag\\
\leq&\, 2C_{6,3}^{\frac{2}{p}}\cdot R^2\cdot   T^{\frac{1}{2}}\cdot e^{2r x}\cdot\left(1+\|u^g\|^2_{\mathcal C^T_r}\right)<+\infty.\notag
\end{align}
Consequently, since $g_n$ converges weakly to  $g$ in $\mathcal H$,   we have that for any fixed $x\in\RR_+$, $t\in[0,T]$,
\begin{align}\label{limggn0}
\lim_{n\rightarrow+\infty}F_n(t,x) =0.
\end{align}
We claim that $\{F_n(t,x); t\geq0, x\geq0\}_{n\geq1}$ is relatively compact in the space $\left(\mathcal C_r^T,\|\cdot\|_{\mathcal C_r^T}\right)$, or equivalently, $\{e^{-rx}F_n(t,x); t\geq0, x\geq0\}_{n\geq1}$ is relatively compact in $C([0,T]\times\RR_+)$ with respect to the sup-norm $\|f\|_\infty=\sup\limits_{t\in[0,T]}\sup\limits_{x\in\RR_+}|f(t,x)|$. Combining this fact with \eqref{limggn0}, we obtain \eqref{limFn0Crt}. 

Notice that
\begin{equation}\label{F_n_estimate}
\begin{split}
\|F_n\|_{\mathcal C_r^T}
\leq&\,\sup_{t\in [0,T]}\sup_{x\ge0} \int_0^t\int_0^{\infty}G_r(t-s, x,y)e^{-ry}\left|\sigma(y, u^{g}(s,y))\right| \left| \dot{g}_n(s,y)-  \dot{g}(s,y)  \right|dyds \\
\leq&\, \left(\sup_{t\in [0,T]}\sup_{x\ge0} \int_0^t\int_0^{\infty}G_r(t-s, x,y)^2e^{-2ry}\left|\sigma(y, u^{g}(s,y))\right|^2dyds\right)^{\frac{1}{2}}\\
&\,\, \cdot \left(\int_0^T\int_0^{\infty}\left| \dot{g}_n(s,y)-  \dot{g}(s,y)\right|^2dyds   \right)^{\frac{1}{2}}.
\end{split}
\end{equation}
By using the same method as that in the proof of \eqref{4-7est}, we have that for any $t\in[0,T]$,
\begin{align}\label{F_n_estimate1}
&\sup_{t\in [0,T]}\sup_{x\ge0} \int_0^t\int_0^{\infty}G_r(t-s, x,y)^2e^{-2ry}\left|\sigma(y, u^{g}(s,y))\right|^2dyds\notag\\
\leq&\,\sup_{t\in [0,T]}\sup_{x\ge0}R^2\cdot\int_0^t\int_0^{\infty}G_r(t-s, x,y)^2\left(1+e^{-ry}| u^{g}(s,y)|\right)^2dyds\\
\leq&\, C_{6,3}^{\frac{2}{p}}\cdot R^2\cdot   T^{\frac{1}{2}}\cdot \left(1+\|u^g\|_{\mathcal C^T_r}\right)^2<+\infty.\notag
\end{align}
Since $g_n, g\in S_N$, $N\ge1$, by \eqref{F_n_estimate} and \eqref{F_n_estimate1}, we have
\begin{align*}
\sup_{n\geq1}\|F_n\|_{\mathcal C_r^T}\leq&\, C_{6,3}^{\frac{1}{p}}\cdot R\cdot  T^{\frac{1}{4}}\cdot \left(1+\|u^g\|_{\mathcal C^T_r}\right)\cdot\left(\int_0^T\int_0^{\infty}\left| \dot{g}_n(s,y)-  \dot{g}(s,y)\right|^2dyds   \right)^{\frac{1}{2}}
<+\infty.
\end{align*}
On the other hand, we claim that $\{e^{-rx}F_n(t,x); t\geq0, x\geq0\}_{n\geq1}$ is also equi-continuous. Then,  according to the Arzel\`a-Ascoli theorem, we obtain that
 $\{F_n\}_{n\geq1}$ is relatively compact in the space $\left(\mathcal C_r^T,\|\cdot\|_{\mathcal C_r^T}\right)$.

Next, we   prove that $\{e^{-rx}F_n(t,x); t\geq0, x\geq0\}_{n\geq1}$ is equi-continuous. Notice that
\begin{equation}\label{tsxyminus1}
\begin{split}
&\left|e^{-ry}F_n(s,y)-e^{-rx}F_n(t,x)\right|\\
\leq&\,\left|e^{-ry}F_n(s,y)-e^{-rx}F_n(s,x)\right|+\left|e^{-rx}F_n(s,x)-e^{-rx}F_n(t,x)\right|.
\end{split}
\end{equation}
Since $g_n, g\in S_N$,  $N\ge1$, by H\"older's inequality, Condition \ref{con 1} (IV) and Lemma \ref{6_1} (iii), we have that for any $p>4$ and for any $s\in[0,T]$, $x, y\in \RR_+$,
\begin{equation}\label{spacexyminus1}
\begin{split}
&\left|e^{-ry}F_n(s,y)-e^{-rx}F_n(s,x)\right|^p\\
=&\,\left|\int^s_0\int^\infty_0\left(G_r(s-u, y,z)-G_r(s-u, x,z)\right)e^{-rz}\sigma(z, u^{g}(u,z))\left( \dot{g}_n(u,z)- \dot{g}(u,z)  \right)dzdu\right|^p\\
\leq&\,\left(\int^s_0\int^\infty_0\left|G_r(s-u, y,z)-G_r(s-u, x,z)\right|^2e^{-2rz}|\sigma(z, u^{g}(u,z))|^2dzdu\right)^{\frac{p}{2}}\\
&\,\cdot\left|\int^s_0\int^\infty_0\left( \dot{g}_n(u,z)- \dot{g}(u,z)  \right)^2dzdu\right|^{\frac{p}{2}}\\
\leq&\,R^p\cdot N^p\left(\int^s_0\left(\int^\infty_0\left|G_r(s-u, y,z)-G_r(s-u, x,z)\right|^2dz\right)\left(1+\|u^{g}\|_{u,\mathcal{L}_r}\right)^2du\right)^{\frac{p}{2}}\\
\leq&\,R^p\cdot N^p\left(\int^s_0\left(\int^\infty_0\left|G_r(s-u, y,z)-G_r(s-u, x,z)\right|^2dz\right)^{\frac{p}{p-2}}du\right)^{\frac{p-2}{2}}\\
&\,\,\,\,\, \cdot\left(\int^s_0\left(1+\|u^{g}\|_{u,\mathcal{L}_r}\right)^pdu\right)\\
\leq&\,C_{6,5}\cdot R^p\cdot N^p\cdot T\cdot
\left(1+\|u^g\|_{\mathcal C^T_{r}}\right)^p|x-y|^{\frac{p-4}{2}}.
\end{split}
\end{equation}
Similarly, by H\"older's inequality, Condition \ref{con 1} (IV) and Lemma \ref{6_1} (ii), we have that for any $p>4$ and for any $s, t\in[0,T]$, $x\in \RR_+$,
\begin{align}\label{timeminus1}
\left|e^{-rx}F_n(s,x)-e^{-rx}F_n(t,x)\right|^p\leq\,C_{6,4}\cdot R^p\cdot N^p\cdot T\cdot\left(1+\|u^g\|_{\mathcal C^T_{r}}\right)^p|s-t|^{\frac{p-4}{4}}.
\end{align}
Putting \eqref{tsxyminus1}, \eqref{spacexyminus1} and \eqref{timeminus1} together, we have that  for any $p>4$ and for any $s, t\in[0,T]$, $x, y\in \RR_+$,
\begin{align*}
\left|e^{-ry}F_n(s,y)-e^{-rx}F_n(t,x)\right|^p\leq\,C_{4,3}\left(1+\|u^g\|^p_{\mathcal C^T_{r}}\right)\cdot\left(|s-t|^{\frac{p-4}{4}}+|x-y|^{\frac{p-4}{2}}\right),
\end{align*}
 where $C_{4,3}=C_{4,3}(r,p,N,R,T)\in(0,\infty)$.

 The proof is complete.
\end{proof}

\section{Verification of Condition \ref{cond1} $(b)$}
For any $\e>0$, let  $\Gamma^{\e}:  C([0,T]\times\RR_+) \rightarrow  \mathcal C_r^T$  be the map defined by
 \begin{align}\label{GAmmae_definition}
 \Gamma^{\e}\left(W(\cdot,\cdot)\right):=u^\e,
 \end{align}
  where $u^\e$ stands for the solution to  Eq.\,\eqref{SPDE}.

Let $\{g^{\e}\}_{\e>0}\subset \mathcal A_N$, $N\geq 1$, be a given family of stochastic processes. Denote
$$\bar{u}^{\e} :=\Gamma^0\left(\int_0^{\cdot}\int_0^{\cdot} \dot{g}^{\e}(s,x)dxds\right),$$ where $\Gamma^{0}$ is defined by \eqref{eq Gamma0}.
Then, $\bar{u}^{\e}$, together with a random measure $\bar{\eta}^{\e}$, solves the following equation:
  \begin{equation}\label{SPDE g e ske}
\left\{\begin{split}
   &\frac{\partial \bar{u}^{\e}}{\partial t}(t, x)=  \frac{\partial^2 \bar{u}^{\e}}{ \partial x^2} (t,x) + f(x,\bar{u}^{\e}(t,x))
  + \sigma(x,\bar{u}^{\e}(t,x))  \dot{g}^{\e}(t,x) +\bar{\eta}^{\e}(dx,dt);\\
   & \bar{u}^{\e}(0,\cdot)=u_0;\\
    & \bar{u}^{\e}(\cdot,0)=0.
 \end{split}\right.
\end{equation}
Denote $$\tilde{u}^{\e}:=\Gamma^{\e}\left(W({\cdot})+\frac{1}{\sqrt \e} \int_0^{\cdot}\int_0^{\cdot}{\dot{g}}^{\e}(s,x)dxds\right),$$
 where $\Gamma^{\e}$ is defined by \eqref{GAmmae_definition}. By the Girsanov theorem, it is easily to see that
$\tilde{u}^{\e}$, together with a random measure $\tilde{\eta}^{\e}$, solves the following stochastic obstacle problem:
 \begin{equation}\label{SPDE g e}
\left\{\begin{split}
  &\frac{\partial \tilde{u}^{\e}}{\partial t}(t, x)= \frac{\partial^2 \tilde{u}^{\e}}{ \partial x^2} (t,x) + f(x,\tilde{u}^{\e}(t,x))+\sqrt{\e} \sigma(x,\tilde{u}^{\e}(t,x)) \dot W(t, x)\\
   &\qquad\quad\quad \quad+ \sigma(x,\tilde{u}^{\e}(t,x)) \dot{g}^{\e}(t,x) +\tilde{\eta}^{\e}(dx,dt);\\
   & \tilde{u}^{\e}(0,\cdot)=u_0;\\
    & \tilde{u}^{\e}(\cdot,0)=0.
 \end{split}\right.
\end{equation}
 \begin{proposition}\label{proposition-4-3} Assume  that Condition \ref{con 1} holds with some $\delta>0$.
Then, for every $N\geq 1$, $\{g^\e\}_{\e>0}\subset \mathcal A_N$ and for any $p\geq1$, we have
    $$\lim_{\e\rightarrow 0}\mathbb E\left(\left \|\tilde{u}^{\e}-\bar{u}^{\e}\right\|_{\mathcal C_{r}^T}^p \right)=0. $$
    \end{proposition}
    Before proving Proposition \ref{proposition-4-3}, we give the following lemma in preparation.
    \begin{lemma}\label{lemma-5-1} Assume  that Condition \ref{con 1} holds with some $\delta>0$.
Then, for every $N\geq 1$, $\{g^\e\}_{\e>0}\subset \mathcal A_N$ and for any $p\geq 1$, we have
    \begin{align}\label{lemma5.1resulit}
    \mathbb{E}\left[\|\tilde{u}^{\e}\|_{\mathcal C_{r}^T}^p\right]<+\infty \ \ \  \text{and }\ \  \mathbb{E}\left[\|\bar{u}^{\e}\|_{\mathcal C_{r}^T}^p\right]<+\infty.
    \end{align}
    \end{lemma}
    \begin{proof}  
     Let
\begin{equation}\label{SPDE g etildev123}
 \begin{split}
  \tilde{v}^\e(t,x)=&\, \int_0^{\infty}G(t, x,y) u_0(y)dy +\int_0^{t}\int_0^{\infty}G(t-s, x,y)f(y, \tilde{u}^\e(s,y)) dyds\\
&\, +\sqrt{\e}\int_0^{t}\int_0^{\infty}G(t-s, x,y)\sigma(y, \tilde{u}^\e(s,y)) W(dy,ds)\\
 &\, +\int_0^{t}\int_0^{\infty}G(t-s, x,y)\sigma(y, \tilde{u}^\e(s,y)) \dot{g}^\e(s,y) dyds\\
 =:&\,J_{1,\e}(t,x)+J_{2,\e}(t,x)+J_{3,\e}(t,x)+J_{4,\e}(t,x).
  \end{split}
  \end{equation}
  Set $\tilde{z}^\e=\tilde{u}^\e-\tilde{v}^\e$. Then $ (\tilde{z}^\e,\tilde{\eta}^\e)$ be the solution of Eq.\,\eqref{eq DOP} with the obstacle $-\tilde{v}^\e$.
   For any $n\geq1$, let
    $$
    \tau_n:=\inf\left\{t\geq0:\sup_{x\geq0}e^{-rx}|\tilde{u}^\e(t,x)|\geq n\right\},
    $$
    and let $J^n_{i,\e}(t,x):=J_{i,\e}(t\wedge\tau_n,x)$, $i=1,\cdots, 4$.
 Since $u_0\in \mathcal{C}_r$, there exists a positive constant $C_{5,1}=C_{5,1}(r,p,T)$ such that
\begin{equation}\label{Je1}
\begin{split}
\mathbb{E}[\|J^n_{1,\e}\|_{\mathcal C_r^T}^p]
=&\,\mathbb{E}\left[\sup_{t\in[0,T]}\sup_{x\geq0}\left|\int_0^\infty G_r(t\wedge \tau_n,x,y)e^{-ry}u_0(y)dy\right|\right]^p\\
\leq&\,\mathbb{E}\left[\sup_{t\in[0,T]}\sup_{x\geq0}\left|\int_0^\infty G_r(t\wedge \tau_n,x,y)dy\cdot\|u_0\|_{\mathcal{L}_r}\right|\right]^p\\
\leq&\,C_{5,1}\mathbb{E}\left[\|u_0\|_{\mathcal{C}_r}^p\right],
\end{split}
\end{equation}
where in the above last inequality we used $\int_0^\infty G_r(t,x,y)dy\leq C_{5,2}(r,t)\leq C_{5,2}(r,T)$ for any $t\in[0,T]$ from (4.8) in \cite{HK2019}. 
By Lemma \ref{lem int}, Condition \ref{con 1} (II) and H\"older's inequality, we have that for any $p>4$,
\begin{equation}\label{Je2}
\begin{split}
\mathbb{E}\left[\|J^n_{2,\e}\|_{\mathcal C_r^T}^p\right] \le &\, C^p_{6, 1}\cdot\mathbb{E}\left[\int_0^T\|f(\cdot,\tilde{u}^\e)\|_{{t\wedge\tau_n}, \mathcal L_r} dt \right]^p\\
 \le &\,C^p_{6, 1}\cdot C^p_{1, 2} \cdot \mathbb{E}\left[\int_0^T\left(1+\|\tilde{u}^\e\|_{{t\wedge\tau_n}, \mathcal L_r} \right)dt \right]^p\\
\le &\, C^p_{6, 1}\cdot C^p_{1, 2}\cdot T^{p-1} \cdot\mathbb{E}\left[\int_0^T\left(1+\|\tilde{u}^\e\|_{\mathcal C_r^{t\wedge\tau_n}}\right)^pd{t}\right].
\end{split}
\end{equation}
Since $\tilde{u}^\e \in  L^p(\Omega; L^\infty([0,T\wedge\tau_n]; \mathcal L_r ))$, by Condition \ref{con 1} (IV), we have
\begin{align*}
\sigma(\cdot, \tilde{u}_n^\e({\cdot\wedge\tau_n},\cdot))\in L^p(\Omega; L^\infty([0,T\wedge\tau_n]; \mathcal L_{r-\delta} )).
\end{align*}
Thus, by Lemma \ref{lem int4} Condition \ref{con 1} (IV), we have that for any $p>12$,
   \begin{equation}\label{Je3}
   \begin{split}
 \mathbb{E}\left[\|J^n_{3,\e}\|_{\mathcal C_r^T}^p\right]
 \leq&\,\e^{\frac{p}{2}}\cdot C_{6, 2}\cdot\mathbb{E}\left[\int_0^T\|\sigma(\cdot, \tilde{u}^\e)\|^p_{{t\wedge\tau_n},\mathcal L_{r-\delta}}dt\right]\\
  \leq&\,\e^{\frac{p}{2}}\cdot C_{6, 2}\cdot R^p\cdot\mathbb{E}\left[\int_0^T\left(1+\|\tilde{u}^\e\|_{{t\wedge\tau_n}, \mathcal L_r} \right)^pdt\right]\\
  \leq&\,\e^{\frac{p}{2}}\cdot C_{6, 2}\cdot R^p\cdot\mathbb{E}\left[\int_0^T\left(1+\|\tilde{u}^\e\|_{\mathcal C_{r}^{t\wedge\tau_n}} \right)^pd{t}\right].
\end{split}
\end{equation}
As in \eqref{smallbound310}, by Condition \ref{con 1} (IV), we have that for any $p>4$,
\begin{align*}
 &\left| e^{-rx} \int_0^{t\wedge\tau_n}\int_0^{\infty}G({t\wedge\tau_n}-s, x, y)  \sigma(y, \tilde{u}^\e(s,y)) )   \dot{g}^\e(s,y) dyds \right|^p\\
\le &R^p\cdot  \left( \int_0^{t\wedge\tau_n}\left(\int_0^{\infty} G_r({t\wedge\tau_n}-s, x, y)^2dy\right)^{\frac{p}{p-2}} ds\right)^{\frac{p-2}{2}}\cdot \int_0^{t\wedge\tau_n}\left(1+\left\|\tilde{u}^\e\right\|_{s,\mathcal L_r}\right)^{p}ds\\
&\quad\cdot
\left| \int_0^T\int_0^{\infty}{ \dot{g}}^\e(s,y)^2 dyds \right|^{\frac{p}{2}}.
  \end{align*}
Since $\{g^{\e}\}_{\e>0}\subset \mathcal A_N$ for some $N\geq 1$,  by Lemma \ref{6_1} (i), we have that  for any $p>4$,
\begin{equation}\label{Je4}
\begin{split}
 \mathbb{E}\left[\|J^n_{4,\e}\|_{\mathcal C_r^T}^p\right]
\le &\, C_{5,3}\cdot \mathbb{E}\left[ \int_0^T\left(1+\left\|\tilde{u}^\e\right\|_{\mathcal C_r^{t\wedge\tau_n}}\right)^{p}d{t}\cdot
\left| \int_0^T\int_0^{\infty}{ \dot{g}}^\e(s,y)^2 dyds \right|^{\frac{p}{2}}\right]\\
\le &\, C_{5,3}\cdot N^{p}\cdot \mathbb{E}\left[ \int_0^T\left(1+\left\|\tilde{u}^\e\right\|_{\mathcal C_r^{t\wedge\tau_n}}\right)^{p}d{t}\right],
\end{split}
\end{equation}
where $C_{5,3}=C_{6,3}\cdot R^p\cdot T^{\frac{p-4}{4}}$.

Putting \eqref{SPDE g etildev123}, \eqref{Je1}, \eqref{Je2}, \eqref{Je3} and \eqref{Je4} together, by Gronwall's inequality, we obtain that there exists a constant $C_{5,4}=C_{5,4}\left(\e,r, p, N,T,\|u_0\|_{\mathcal C_r}\right)$ such that
\begin{align*}
\mathbb{E}\left[\|\tilde{v}^{\e}\|_{\mathcal C_r^{T\wedge \tau_n}}^p\right]\leq\,C_{5,4},
\end{align*}
where $C_{5,4}$ is independent on $n$. Letting $n\rightarrow\infty$, we have 
\begin{align*}
\mathbb{E}\left[\|\tilde{v}^{\e}\|_{\mathcal C_r^T}^p\right]\leq\,C_{5,4}.
\end{align*}
By Theorem \ref{thm DOP}, we have 
\begin{align*}
\mathbb{E}\left[\|\tilde{u}^{\e}\|_{\mathcal C_r^T}^p\right]\leq 2^p\cdot C^p_{3,1}\mathbb{E}\left[\|\tilde{v}^{\e}\|_{\mathcal C_r^T}^p\right]<\infty.
\end{align*}

The proof  for $\bar u^{\e}$ is similar and is omitted here. The proof is complete.
    \end{proof}

    \begin{proof}[Proof of Proposition \ref{proposition-4-3}]
   Recall $\tilde{v}^\e(t,x)$ defined by \eqref{SPDE g etildev123}, and  let
 \begin{equation}\label{SPDE g ebarv}
 \begin{split}
  \bar{v}^\e(t,x)=&\, \int_0^{\infty}G(t, x,y) u_0(y)dy +\int_0^t\int_0^{\infty}G(t-s, x,y)f(y, \bar{u}^\e(s,y)) dyds\\
 &\, +\int_0^t\int_0^{\infty}G(t-s, x,y)\sigma(y, \bar{u}^\e(s,y)) \dot{g}^\e(s,y) dyds.
  \end{split}
  \end{equation}
 Then $(\tilde{u}^\e-\tilde{v}^\e,\tilde{\eta}^\e)$ and $(\bar{u}^\e-\bar{v}^\e,\bar{\eta}^\e)$ are the solutions of Eq.\,\eqref{eq DOP} with  the   obstacles $-\tilde{v}^\e$ and $-\bar{v}^\e$, respectively.

By  Theorem \ref{thm DOP},    \eqref{SPDE g e ske}, \eqref{SPDE g e}, \eqref {SPDE g etildev123} and \eqref{SPDE g ebarv}, we obtain that
\begin{align}\label{tildeue-barue22}
& \|\tilde{u}^\e-\bar{u}^\e\|_{\mathcal C_r^T}\notag\\
\le&\,  2 C_{3, 1}\|\tilde{v}^\e-\bar{v}^\e\|_{\mathcal C_r^T}\notag\\
\le & \,  2  C_{3, 1} \sup_{t\in [0,T]}\sup_{x\ge0} e^{-rx} \left|\int_0^t\int_0^{\infty}G(t-s, x,y) \left( f(y, \tilde{u}^\e(s,y))- f(y, \bar{u}^\e(s,y))  \right) dyds \right|\notag\\
 &\,+2C_{3, 1} \sup_{t\in [0,T]}\sup_{x\ge0} e^{-rx} \left|\sqrt{\e}\int_0^t\int_0^{\infty}G(t-s, x,y)\sigma(y, \tilde{u}^\e(s,y)) W(dy,ds) \right|\\
&\,+ 2  C_{3, 1} \sup_{t\in [0,T]}\sup_{x\ge0} e^{-rx} \left|\int_0^t\int_0^{\infty}G(t-s, x,y) \left( \sigma(y, \tilde{u}^\e(s,y))- \sigma(y, \bar{u}^\e(s,y))  \right) \dot{g}^\e(s,y) dyds \right|\notag\\
=:&\, 2  C_{3, 1}\left[Q_{1,\e}(T)+Q_{2,\e}(T)+Q_{3,\e}(T)\right].\notag
\end{align}
As in  \eqref{I_1p11}, we have that for any $p>4$,
\begin{equation}\label{I_1p22}
\begin{split}
\mathbb{E}\left[Q_{1,\e} (T)^p\right]
\le \,  C^p_{6, 1}\cdot C^p_{1, 1}\cdot T^{p-1}\cdot\mathbb{E}\left[ \int_0^T\|\tilde{u}^\e-\bar{u}^\e\|^p_{\mathcal C_r^t}dt\right].
\end{split}
\end{equation}
Since $\tilde{u}^\e \in  L^p(\Omega;\mathcal C_r^T )$, by Condition \ref{con 1} (IV), we have
\begin{align*}
\sigma(\cdot, \tilde{u}^\e(\cdot,\cdot))\in L^p(\Omega;\mathcal C_{r-\delta}^T).
\end{align*}
As in \eqref{Je3},  by Lemma \ref{lemma-5-1}, we have that for any $p>12$,
   \begin{equation}\label{I_1p221}
   \begin{split}
   \mathbb{E}[Q_{2,\e}(T)^p]
  \leq&\,\e^{\frac{p}{2}}\cdot C_{6, 2}\cdot R^p\cdot\mathbb{E}\left[\int_0^T\left(1+\|\tilde{u}^\e\|_{\mathcal C_{r}^t}\right)^pdt\right]
  \rightarrow0,\,\,\emph{as}\,\,\e\rightarrow0.
\end{split}
\end{equation} 
As in \eqref{Je4},  by Condition \ref{con 1} (III), we have that for any $p>4$, 
\begin{equation}\label{I_3p22}
\begin{split}
Q_{3,\e} (T)^p  
\le &\, C^p_{1,3} \cdot C_{6,3} \cdot T^{\frac{p-4}{4}} \cdot N^p\int_0^T\|\tilde{u}^\e-\bar{u}^\e\|_{\mathcal C_r^t}^pdt.
\end{split}
\end{equation}
Putting \eqref{tildeue-barue22}, \eqref{I_1p22},  \eqref{I_1p221} and  \eqref{I_3p22} together,   by the Gronwall inequality, we obtain that for any $p>12$,
  \begin{align*}
 \mathbb E\left[\|\tilde{u}^\e-\bar{u}^\e\|^p_{\mathcal C_r^T}\right]\leq \mathbb E\left[Q_{2,\e} (T) ^p\right]\left(1+ C_{5,5}T \cdot e^{  C_{5,5}T }\right)\rightarrow0, \ \ \text{as } \e\rightarrow0,
 \end{align*}
 where $C_{5,5}= 2^p\cdot C^p_{3, 1}\cdot C^p_{6, 1}\cdot C^p_{1, 1}\cdot T^{p-1}+C^p_{1,3} \cdot C_{6,3} \cdot T^{\frac{p-4}{4}} \cdot N^p$. 

 The proof is complete.
    \end{proof}

\section{Appendix}
Recall   the  functions $G(t,x,y)$   and $G_r(t,x,y)$   given by \eqref{eq kernel} and \eqref{eq Gr}, respectively. 
\begin{lemma}\cite[Lemma 4.5]{HK2019}\label{lem int}
Let $r\in \mathbb R$. Suppose that $u\in L^1([0,T];\mathcal L_r)$. Then we have that for any $t\in [0,T]$, there exists a positive constant $C_{6, 1}=C_{6, 1}(r,T)$ such that
 \begin{align*}
 \sup_{\tau\in[0,t]}\sup_{x\ge0}\left|e^{-rx}\int_0^{\tau}\int_0^{\infty} G(\tau-s, x,y)u(s,y)dyds \right|\le C_{6, 1}\int_0^t\|u\|_{s, \mathcal L_r}ds,
 \end{align*}
 \end{lemma}

\begin{lemma}\cite[Proposition 4.8]{HK2019}\label{lem int4} Let $r\in \mathbb R$. Suppose that $u\in L^p(\Omega; L^{\infty}([0,T]; \mathcal L_{r-\varepsilon}))$. Then for any $t\in [0,T]$ and for any $p>12$, there exists a positive constant $C_{6, 2}=C_{6, 2}(r,p, T,\varepsilon)$ such that
\begin{align*}
& \mathbb E\left[ \sup_{\tau\in[0,t]}\sup_{x\ge0}\left|e^{-rx}\int_0^{\tau}\int_0^{\infty} G(\tau-s, x,y)u(s,y)W(dy,ds) \right|^p\right]\\
\le&\,  C_{6,2}\mathbb E\left[ \int_0^t\|u\|_{s, \mathcal L_{r-\varepsilon}}^pds\right].
\end{align*}
 \end{lemma}
\begin{lemma}\cite[Proposition A.1]{HK2019}\label{6_1} Fix  $r\in \mathbb R$, $T>0$. Then, for any $p>4$, we have the following results:
\begin{itemize}
\item[(i)] For every $t, s\in[0,T]$, there exists a positive constant $C_{6, 3}=C_{6, 3}(r,p, T)$ such that
\begin{align*}
\sup_{x\ge0}\left( \int_s^t\left[\int_{0}^{\infty}G_r(t-u, x,z)^2dz \right] ^{\frac{p}{p-2}}du \right)^{\frac{p-2}{2}}\le C_{6, 3}|t-s|^{\frac{p-4}{4}}.
\end{align*}
\item[(ii)] For every $t, s\in[0,T]$, there exists a positive constant $C_{6, 4}=C_{6, 4}(r,p, T)$ such that
\begin{align*}
\sup_{x\ge0}\left( \int_0^s\left[\int_{0}^{\infty}\left|G_r(t-u, x,z)-G_r(s-u, x,z)\right|^2dz \right] ^{\frac{p}{p-2}} du \right)^{\frac{p-2}{2}}\le C_{6,4}|t-s|^{\frac{p-4}{4}}.
\end{align*}
\item[(iii)] For every $x,y\in[0,\infty)$, there exists a positive constant $C_{6, 5}=C_{6, 5}(r,p, T)$ such that
\begin{align*}
\sup_{s\in[0,T]}\left( \int_0^s\left[\int_{0}^{\infty}\left|G_r(s-u, x,z)-G_r(s-u, y,z)\right|^2dz \right] ^{\frac{p}{p-2}} du \right)^{\frac{p-1}{2}}\le C_{6,5}|x-y|^{\frac{p-4}{2}}.
\end{align*}
\end{itemize}
\end{lemma}

\vskip0.5cm
\noindent{\bf Acknowledgments}:    The research of  R. Wang  is partially supported by  NNSFC grants 11871382 and 12071361. The research of B. Zhang is
partially supported by NNSFC grants 11971361 and 11731012.

\medskip

\bigskip


\vskip0.5cm



\begin{thebibliography}{abc}

\bibitem{DMZ2006}
Dalang R C,  Mueller C, Zambotti L. Hitting properties of parabolic SPDE's with reflection.  Ann Probab, 2006, {34}(4): 1423-1450

\bibitem{DZ2013}
Dalang R C,
  Zhang T S. H\"older continuity of solutions of SPDEs with reflection. Commun  Math  Statist, 2013, {1}(2): 133-142

 \bibitem{DMP1993}
Donati-Martin C, Pardoux E. White noise driven SPDEs with reflection. Probab Theory Relat  Fields, 1993, {95}: 1-24



\bibitem{FO2001}
 Funaki T,  Olla S. Fluctuations for $\nabla\varphi$ interface model on a wall.  Stochastic Process Appl, 2001, 94(1):  1-27

 \bibitem{HK2019}
 Hambly B,  Kalsi J. A reflected moving boundary problem driven by space-time white noise. Stoch Partial Differ Equ Anal Comput,  2019, {7}(4):  746-807 


 \bibitem{Kalsi2020}
 Kalsi J. Existence of invariant measures for reflected stochastic partial differential equations. J Theoret Probab, 2020, {33}(3): 1755-1767


\bibitem{MSZ}
  Matoussi A, Sabbagh W,  Zhang T S. Large deviation principles of obstacle problems for
quasilinear stochastic PDEs.  Appl Math Optim, 2021, {83}(2): 849-879 

\bibitem{NP1992}
Nularat E, Pardoux E. White noise driven by quasilinear SPDEs with reflection. Probab Theory Relat Fields, 1992, {93}: 77-89

 \bibitem{Oto2002}
Otobe Y. Stochastic reaction diffusion equations on an infinite interval with reflection.  Stoch Stoch Rep, 2002, {74}(1-2): 489-516 


\bibitem{WJW2020}
 Wang S, Jiang Y, Wang Y. Stochastic partial differential equation with reflection driven by fractional noises. Stochastics, 2020, {92}(1): 46-66

 \bibitem{X2019}
 Xie B. Hypercontractivity for space-time white noise driven SPDEs with reflection.  J  Differential Equations, 2019, {266}(9): 5254-5277

\bibitem{X2022}
 Xie B. Log-Harnack inequality for reflected SPDEs driven by multiplicative noises and its applications. Stoch Partial Differ Equ Anal Comput, 2022, 10:  419-445 

\bibitem{XZ2009}
Xu T,  Zhang T S. White noise driven SPDEs with reflection: existence, uniqueness and large deviation principles. Stoch Process Appl, 2009,  {119}(10): 3453--3470 

 \bibitem{YangZ2014}
Yang J, Zhang T S. Existence and uniqueness of invariant measures for SPDEs with two reflecting
walls. J Theoret Probab, 2014, 27(3): 863-877

 \bibitem{YZ2020}
 Yang J, Zhou Q. Reflected SPDEs driven by fractional noise.  Acta Math Appl Sin Engl Ser, 2020, {36}(2): 347-360


 \bibitem{Zambotti2001}
Zambotti L. A reflected stochastic heat equation as symmetric dynamics with respect to the 3-d Bessel bridge. J Funct Anal, 2001,  {180}: 195-209 

  \bibitem{Zambotti2017}
 Zambotti L. Random Obstacle Problems. Lecture Notes in Mathematics, Vol 2181, Springer, Cham, 2017
\bibitem{Zhang2010}
  Zhang T S. White noise driven SPDEs with reflection: strong Feller properties and Harnack inequalities. Potential Anal, 2010, {33}(2):  137-151

\end{thebibliography}
\end{document}